\documentclass[11pt]{amsart}

\usepackage{amsfonts,amsmath,latexsym,amssymb,verbatim,amsbsy,amsthm}

\usepackage{graphicx}
\usepackage{nicefrac}
\usepackage{dsfont}
\usepackage{mathrsfs}
\usepackage{cancel}
\usepackage[normalem]{ulem}
\usepackage{color}

\usepackage[top=1in, bottom=1in, left=1in, right=1in]{geometry}

\usepackage{pgf,pgfarrows,pgfnodes,pgfautomata,pgfheaps,pgfshade}
\usepackage{fancyhdr}
\usepackage{soul}
\usepackage{wrapfig}
\usepackage{nicefrac}
\usepackage{epsfig}

\usepackage[colorlinks=true, pdfstartview=FitV, linkcolor=blue,citecolor=red, urlcolor=blue]{hyperref}

\numberwithin{equation}{section}
\numberwithin{figure}{section}

\newtheorem{theorem}{Theorem}[section]
\newtheorem{corollary}[theorem]{Corollary}
\newtheorem{lemma}[theorem]{Lemma}
\newtheorem{proposition}[theorem]{Proposition}
\theoremstyle{definition}
\newtheorem{definition}[theorem]{Definition}
\newtheorem{example}[theorem]{Example}
\newtheorem{remark}[theorem]{Remark}

\renewcommand{\geq}{\geqslant}

\renewcommand{\leq}{\leqslant}

\newcommand{\ds}{\displaystyle} 
\newcommand{\be}{\begin{equation}}
\newcommand{\ee}{\end{equation}}

\ifx
\theoremstyle{plain}
\newtheorem{THEOREM}{Theorem}[section]

\newtheorem{theorem}[THEOREM]{Theorem}

\theoremstyle{definition}

\theoremstyle{remark}
\theoremstyle{question}

\newtheorem{remark}[THEOREM]{Remark}

\fi

\def\d{{\textnormal{d}}}
\def\bx{{\mathbf x}}
\def\by{\bx'} 
\def\bz{{\mathbf z}}
\def\bu{{\mathbf u}}
\def\bup{\bu'}
\def\bv{{\mathbf v}}
\def\v{v}

\def\dy{\d\by}

\def\dv{\d\bv}

\def\deg{{\textit{deg}}}
\renewcommand{\geq}{\geqslant} 
\renewcommand{\leq}{\leqslant}  
\def\extif{\ifx}

\def\bbphi{\Phi} 
\def\step{\tau} 
\def\delE{{\delta \mathscr E}}
\def\delV{{\delta \mathscr V}}
\def\hf{\frac{1}{2}}
\def\deg{deg}
\def\erg{\kappa} 
\def\bo{{\boldsymbol \omega}}
\def\Egraph{{\sf E}}
\def\Vgraph{{\sf V}}

\begin{document}
\title[Long time and large crowd dynamics of  Cucker-Smale  models]{Long time and large crowd dynamics\\of fully discrete  Cucker-Smale alignment models}


\author[Eitan Tadmor]{Eitan Tadmor}
\address[Eitan Tadmor]{Department of Mathematics and Institute for Physical Sciences \& Technology (IPST), University of Maryland, College Park}
\email{{\tt tadmor@umd.edu}}

\date{April 1, 2022}

\subjclass{92D25, 70F40, 35Q35.}

\keywords{flocking, alignment, spectral analysis, strong solutions, critical thresholds.}

\thanks{\textbf{Acknowledgment.} Research was supported in part by ONR grant N00014-2112773.}

\dedicatory{\bigskip{\large To Ron \mbox{De\!Vore} for many years of great friendship}}

\begin{abstract}
We provide a bird's eye view on  developments in analyzing the long time, large crowd behavior of Cucker-Smale alignment dynamics. We consider a  class of (fully-)discrete models, paying particular attention to  general alignment protocols in which agents, with possibly time-dependent masses, are driven by a  large class of heavy-tailed communication kernels.  The presence of time-dependent masses allows, in particular, non-symmetric communication. While revisiting known results in the literature, we also  shed  new light of various aspects on the long time  flocking/swarming  behavior, driven by the decay of energy fluctuations and heavy-tailed connectivity.  We also discuss the large crowd dynamics in terms of the hydrodynamic description of the corresponding Euler alignment models.
\end{abstract}

\maketitle
\setcounter{tocdepth}{2}
\tableofcontents

\section{The Cucker-Smale model}
{\LARGE I}\hspace*{-0.051cm}n 1998, Craig Reynolds won  a Scientific and Engineering Award of the Academy of Motion Picture Arts and Sciences for ``\emph{pioneering contributions to the development of three dimensional computer animation for motion picture production}'', \cite{Rey98}.    Reynolds was recognized for his work on realistic simulations of \emph{flocking}, \cite{Rey87},  proposing a  collective dynamics   of `bird-like objects'   (or `boids') which are driven by pairwise interactions acting in three  zones of \emph{repulsion}, \emph{alignment} and \emph{attraction}. 
A  similar 3Zone protocol  is found in  a broad spectrum  of models for collective dynamics in different contexts: in modeling swarming dynamics in \emph{ecology} --- from fish, birds  and sheep to bacteria,  locust and  insects, 
 \cite{Aok82,Par82,Rey87,BCE91,KL93,VCBCS95,TT98,EKWG98,BDT99,WB01,
 EK01,CKJRF02,PVG02,Ben03,CF03,TB04,GR05,HCK05,CKFL05,DCBC06,
 OAGM06,CS07a,CS07b,Bal08,HH08,YBTBG08,BEBS09,LX10,JK10,
 Cav10,SASBJ11,CCGPS12,Bia12,GWBL12,TDOEKB12,Bia14,
 AA15,Gin15,Fon16,Jia17,PT17,Ari18,Cal18,USTB19,Liu20}; modeling  social dynamics of \emph{human interactions} --- from pedestrians, exchange of opinions and ratings  to markets and marketing, \cite{DeG74,Axe84,Axe97,Kra00,Hel01,HK02,WDA05,BHOT05,BeN05,
 Wei06,Lor07,BHT09,CFL09,Hel09,FG10,Hel10,PT11,BT15,RDW18,BCD19}; and in modeling the dynamics of \emph{sensor-based networks} --- ranging from macro-molecules and metallic rods to  control and  mobile robot networks, \cite{Kur75,OK91,JLM03,OSM04,CMB06,OS06,JE07,ZP07,ZEP11,Rin12,BV13,MT14b}. 
The common theme of the different models is crowd dynamics dictated by  pairwise interactions between members of the crowd which are  viewed as \emph{agents}.
A main question of interest is to understand how the small scale pairwise interactions within the crowd, are self-organized into a large scale patterns of the whole crowd, so that ``the whole is greater than the sum of its parts''. One then refers to the \emph{emergent behavior} of the crowd, where the larger patterns are  realized by a crowd forming a flock, reaching a consensus, admitting a synchronized state, aggregate into one or more clusters, etc.

\smallskip\noindent
\paragraph{{\bf The class Cucker-Smale alignment models}} Pairwise attraction and repulsion are familiar from particle physics, for example, particle dynamics driven Coulomb and other singular potentials, \cite{ST97,Ser17,LS17,Ser18,LS18}. Here, we focus our attention on \emph{alignment} dynamics, driven by  pairwise interactions in which  agents steer towards \emph{average heading}. We consider the  agent-based system in which  $N$ agents, identified  with (position, velocity) pairs $(\bx_i(t),\bv_i(t)):  {\mathbb R}_+ \mapsto (\Omega,{\mathbb R}^d)$ and subject to prescribed initial conditions, $(\bx_i(0),\bv_i(0))=(\bx_{i0},\bv_{i0})\in (\Omega,{\mathbb R}^d)$, are driven by
\begin{equation}\label{eq:CS}
\left\{ \ \ \begin{split}
{\bx}_i(t+\step)&=\bx_i(t)+\step\bv_i(t)\\
{\bv}_i(t+\step)&=\bv_i(t) + \step\sum_{j\in {\mathcal N}_i} m_j\phi_{ij}(t)(\bv_j(t)-\bv_i(t)). 
\end{split}\right.
\end{equation}
The dynamics is dictated by  a  symmetric  \emph{communication kernel},  
\[
\phi(\bx,\by)=\phi(\by,\bx)\geq 0.
\]
 Its dynamic values, $\phi_{ij}(t)=\phi(\bx_i(t),\bx_j(t))$,
 encode the `rule of engagement' between agents, and in particular  the neighborhood ${\mathcal N}_i=\{j : \phi_{ij}(t)>0\}$,  which  contributes to the steering of a `boid' positioned at $\bx_i$.
The spatial domain $\Omega$ is either ${\mathbb T}^d$ or ${\mathbb R}^d$, so that boundaries are avoided, and $\step>0$ is a small, possibly variable time-step, $\step=\step(t)$. Different agents, $(\bx_i,\bv_i)$, are assumed to have different masses, $m_i$, or other constant traits attributed to an  agent positioned at $\bx_i$.\newline 
We refer to \eqref{eq:CS} as the class of Cucker-Smale (C-S) models for alignment dynamics. Different models are attached to different $\phi$'s and different $m_i$'s.
The original model of Cucker \& Smale (C-S) \cite{CS07a,CS07b}  is \emph{the} canonical model for the class of alignment dynamics \eqref{eq:CS}  with $\phi(\bx,\bx')\sim (1+|\bx-\bx'|)^{-\beta}, \ \beta>0$, which assumes a uniform mass distribution $m_i\equiv \nicefrac{1}{N}$,
\begin{equation}\label{eq:equiCS}
{\bv}_i(t+\step)=\bv_i(t) + \frac{\step}{N}\sum_{j\in {\mathcal N}_i} \phi_{ij}(t)(\bv_j(t)-\bv_i(t)).
\end{equation}
The work of Cucker \& Smale   attracted a considerable attention in the literature and motivated the study of many variants of the C-S alignment models; we refer to  \cite{She07,PKH10,Pes15,BDT17-19,Tad17,CHL17, CHJK19,Shv21} and the references therein.
In particular, a more general alignment model based on the formation of  `blobs' or multi-flocks of agents with different masses was derived  in \cite{ST21b}. In other models, different $m_i$'s can be identified with different intrinsic `traits' of different agents, such as  degree,  temperature, \cite{MT11, HKR18, Jin18, CHJK19, Has21}. We further elaborate on one example.

\smallskip\noindent
\paragraph{{\bf The Motsch-Tadmor model}} 
If each of the  terms contributing to the C-S alignment on the right of \eqref{eq:equiCS},  $ \sum_j \phi_{ij}(\bv_j-\bv_i)$,  is of the same  ${\mathscr O}(1)$-order, then its total action of order ${\mathscr O} (N)$ will peak at time $t = {\mathscr O} (\nicefrac{1}{N})$. Thus, as noted in \cite[\S2] {ST21b}, the pre-factor $\nicefrac{1}{N}$ is C-S model \eqref{eq:equiCS} is in fact a \emph{scaling} factor, so that the dynamics peaks at the desired time $t \sim {\mathscr O} (1)$.\newline
In \cite{MT11} we advocated a more realistic  scaling which is adapted to 
spatial variability in the intensity of different alignment terms,
\begin{equation}\label{eq:MT}
{\bv}_i(t+\step)=\bv_i(t) + \frac{\step}{\sum \limits_{k\in {\mathcal N}_i} \phi_{ik}(t)}\sum_{j\in {\mathcal N}_i} \phi_{ij}(t)(\bv_j(t)-\bv_i(t)).
\end{equation}
Here the scaling depends on the \emph{degree} of  different agents, 
\[
\deg_i:=\sum_{k\in {\mathcal N}_i} \phi_{ik}(t).
\]
It should be  emphasized that the   communication array in M-T model, $\ds \Big\{\frac{1}{\deg_i}\phi_{ij}\Big\}$ is not symmetric. Nevertheless, it does fit  the general  symmetric framework of C-S class \eqref{eq:CS} with a proper choice of `masses'
$\displaystyle m_i =\frac{1}{L}\deg_i$, and  \emph{symmetric} interactions
$\ds \widetilde{\phi}_{ij}= L\phi_{ij}\frac{1}{\deg_i}\frac{1}{\deg_j}$, recovering \eqref{eq:MT},
\begin{equation}\label{eq:MTS}
{\bv}_i(t+\step) = \bv_i(t) + \step\sum_{j\in {\mathcal N}_i} m_j\widetilde{\phi}_{ij}(t)(\bv_j(t)-\bv_i(t)), \quad m_i=\frac{1}{L}\deg_i, \quad \widetilde{\phi}_{ij}= \frac{1}{L}\phi_{ij}\frac{1}{m_i}\frac{1}{m_j}
\end{equation}
The scaling parameter $L$ has no effect on the alignment and was introduced here in order to re-scale the total mass\footnote{\label{foot:MT}For example,, in the case of long range all-to-all communication where $\phi_{ij}={\mathcal O}(1)$, then $\deg_j={\mathcal O}(N)$ and we set $L=N^2$ so that $M=\frac{1}{L}\sum_j\deg_j ={\mathcal O}(1)$.} so that $M:=\sum_i m_i ={\mathcal O}(1)$.
In this  case, however, the degrees vary in time, $\ds m_i=\frac{1}{L}\deg_i(t)$ and the discussion below needs to be  modified to include time-dependent masses. This will be further explored in section \ref{sec:MT} below.\newline
As another example, we mention a similar situation that arises in  the  context of \emph{thermodynamic} C-S model \cite{HKR18,CHJK19}, where $m_i$'s can be identified with the different temperatures $m_i=\theta_i(t)$ of agents with re-scaled velocities 
$\displaystyle \frac{1}{\theta_i}\bv_i$. Again, one needs to address the time-dependence of the temperatures which are dictated by a separate dynamics.

\section{Communication kernels}
The  dynamics of \eqref{eq:CS} is dictated by  a  symmetric  \emph{communication kernel},  
$\phi(\cdot,\cdot)\geq 0$.
Where do these communication kernels come from? they arise from a combination of empirical and phenomenological considerations.
 A sample of the large literature  can be found in 
\cite{Kur75,TT98,Kra00,WB01,CF03,JLM03,CKFL05,GR05,WDA05,Wei06,ZP07,Jac10,ZEP11,SASBJ11,VZ12,Gin15} and the references therein.
We mention several primary examples.\newline
A large part of current literature is devoted to the  generic class of  \emph{metric-based} kernels,
\[
\phi(\bx,\by)=\varphi(|\bx-\by|).
\]
The choice of metric kernels $\varphi(r)=\mathds{1}_{[0,R_0]}$ and   $\varphi(r)=(1+ r)^{-\beta}, \ 0<\beta<1$    are found in the seminal works of  Vicsek et. al. \cite{VCBCS95} and respectively Cucker \& Smale \cite{CS07a}. They are  motivated by a \emph{phenomenological} reasoning that the strength of pairwise interactions is  short-range or at least decreasing with the relative distance, ``\emph{birds of feather flock together}'' \cite{MSLC01}; this should be contrasted with an opposite heterophilous protocol, \cite{MT14a}, based on tendency to attract diverse groups so that $\varphi(r)$ is \emph{increasing} over its compact support.
A particular sub-class of such metric-based  protocols are the \emph{singular}   kernels, $\varphi(r)=r^{-\beta},  0<\beta <d+2$,
which   emphasize near-by neighbors, $r\ll 1$, over those farther away,
\cite{Pes15,CCMP17,ST17a,PS17,ST18,DKRT18,MMPZ19}. The case of non-summable  kerenls, $\beta=d+2s, s\in (0,1)$  correspond to Riesz kernels and could be properly interpreted as principle values of  summation in the commutator form \cite{ST17a}
\[
\sum_j \phi_{ij}m_j(\bv_j-\bv_i)
= \sum_j \frac{m_j\bv_j-m_i\bv_i}{|\bx_j-\bx_i|^{d+2s}}
-\sum_j \frac{m_j-m_i\hspace*{0.6cm}}{|\bx_j-\bx_i|^{d+2s}}\bv_i.
\]
 An important source for communication kernels  are  detailed observations.  As a prime example we mention  the class of \emph{topologically-based} kernels,  dictated  by the \emph{size} of the crowd in between agents positioned at $\bx$ and $\by$
 \begin{equation}\label{eq:whatismu}
   \phi(\bx,\by)=\varphi(\mu(\bx,\by)), \quad \mu(\bx,\by):=\frac{1}{N}\#\{k:\, \bx_k\in {\mathcal C}(\bx,\by)\}.
\end{equation}
Here, ${\mathcal C}(\bx,\by)$ is a pre-determined communication region  enclosed between $\bx$ and $\by$. 
In particular, if  ${\mathcal C}$ is shifted to $R$-ball centered at $\bx$, one ends up with the \emph{non-symmetric} topological kernel \cite{MT11}
$\displaystyle \phi(\bx,\by)=\frac{\varphi(|\bx-\by|)}{\mu(B_R(\bx))}$. 
Topologically-based communication was observed in starflag project reported in 
\cite{Bal08,Cav08,Cav10,CCGPS12}, 
where birds react to the number of closest neighbors rather than their metric distance,  and in pedestrian dynamics \cite{RDW18}, 
where communication is decreasing in more crowded regions, and was  analyzed in \cite{Bal08,BD16,ST20b}.\newline 
More on topologically-based kernels can be found in \cite{OSM04,Li08,CCGPS12,Has13,BD17}\newline
As a third example, we mention  \emph{random-based} communication protocols found in chemo- and photo-tactic dynamics,  \cite{HL09a}, 
the Elo rating system, \cite{JJ15,DTW19}, 
voter and related opinion-based models, \cite{BeN05}, 
or a random-batch method and consensus-based optimization.
 \cite{DAWF02,CFL09,GWBL12,PTTM17,JLL20}. 
Another class of communication  kernels are those learned from the data, 
\cite{BPK16,LZTM19,MLK19}.  
Finally, we mention  communication  kernels  which are derived from `higher order' principles; for example, a minimum entropy principle \cite{Bia12,Bia14},
 and the paradigm of anticipation \cite{GTLW17}.

\section{Long time dynamics}
A key aspect in the long time behavior of \eqref{eq:CS} is the decay in time of the fluctuations of velocities $\{\bv_i-\bv_j\}$. Velocity fluctuations can be measured in a weighted-$\ell^2$  average sense quantifying \emph{energy fluctuations},  or in a uniform  sense quantifying the $\ell^\infty$-\emph{diameter} of the discrete crowd of velocities.
\subsection{Energy fluctuations}\label{sec:energy}
 We let $\delE(t)$ denote the \emph{energy fluctuations}, scaled by the total mass\footnote{Here and below,
$|\cdot|$ denotes an arbitrary vector norm on ${\mathbb R}^d$.}
\[
\delE(t):=\frac{1}{2M^2}\sum_{i,j}|\bv_i(t)-\bv_j(t)|^2m_im_j, \qquad M=\sum_i m_j.
\]
Thus, $\delE(t)$ is the weighted $\ell^2$-diameter of the set of velocities $\{\bv_i\}_{i=1}^N$ at time $t$. Equivalently, we can express it as fluctuations around the mean velocity $\overline{\bv}$
\begin{equation}\label{eq:mean}
\delE(t)=\frac{1}{M}\sum_i |\bv_i(t)-\overline{\bv}(t)|^2m_i, \qquad \overline{\bv}:=\frac{1}{\sum_i m_i}\sum_i m_i\bv_i(t)
\end{equation}
The energy balance encoded in \eqref{eq:CS}${}_2$ implies (for simplicity we suppress the time dependence on $t$ on the right-hand side)
\begin{equation}\label{eq:energy}
\left\{\ \begin{split}
\frac{1}{M}\sum_i m_i&|\bv_i(t+\step)|^2 -
\frac{1}{M}\sum_i m_i|\bv_i(t)|^2 \\
 & = \frac{2\step}{M}\sum_i \big\langle m_i\bv_i, \sum_j m_j\phi_{ij}(\bv_j-\bv_i)\big\rangle  + \frac{\step^2}{M}\sum_im_i\big|\sum_j m_j\phi_{ij}(\bv_j-\bv_i)\big|^2.
\end{split}\right.
\end{equation}
Since the communication kernel is symmetric, $\phi_{ij}=\phi_{ji}$, the total momentum is conserved
\begin{equation}\label{eq:momentum}
{\mathscr M}(t+\step)-{\mathscr M}(t) =\frac{\step}{M}\sum_{i,j}m_im_j\phi_{ij}(\bv_j-\bv_i)=0, \qquad {\mathscr M}(t):=\frac{1}{M}\sum_i m_i\bv_i(t).
\end{equation}
This implies that the incremental change in energy of the left of \eqref{eq:energy}
is the same as the incremental change of energy fluctuations. Indeed, 
\[
\frac{1}{M}\sum_i m_i|\bv_i(t)|^2 \equiv \frac{1}{2M^2}\sum_{i,j}|\bv_i(t)-\bv_j(t)|^2m_im_j + \frac{1}{M^2}\big|\sum_i m_i\bv_i(t)\big|^2= \delE(t)+\big|{\mathscr M}(t)\big|^2,
\]
and the same applies at $t+\step$,
\[
\frac{1}{M}\sum_i m_i|\bv_i(t+\step)|^2 \equiv \delE(t+\step)+\big|{\mathscr M}(t+\step)\big|^2,
\]
and since the  squared terms  on the right of the last two equalities are the same, we find
\begin{subequations}\label{eqs:abc}
\begin{equation}\label{eq:abc}
\frac{1}{M}\sum_i m_i |\bv_i(t+\step)|^2 -\frac{1}{M}\sum_i m_i|\bv_i(t)|^2 =\delE(t+\step) - \delE(t).
\end{equation}
 We now come to the main point, namely, that the alignment operator of CS dynamics is coercive in the sense that
\begin{equation}\label{xyz}
\frac{2}{M}\sum_i\big\langle m_i\bv_i, \sum_j m_j\phi_{ij}(\bv_j-\bv_i)\big\rangle = - \frac{1}{M}\sum_{i,j}\phi_{ij}|\bv_i-\bv_j|^2m_im_j.
\end{equation}
The \emph{weighted} fluctuations on the  right is identified as  the \emph{enstrophy}.
We can bound the last squared term on the right of \eqref{eq:energy} in terms of the enstrophy and the maximal \emph{weighted degree},
$\deg_+(t):=\max_i\sum_j \phi_{ij}(t)m_j$
\begin{equation}\label{rst}
\frac{1}{M}\sum_i m_i\big|\sum_j m_j\phi_{ij}(\bv_j-\bv_i)\big|^2 \leq 
 \deg_+(t)\frac{1}{M}\sum_{i,j} \phi_{ij}|\bv_j-\bv_i\big|^2m_im_j;
\end{equation}
\end{subequations}
Inserting \eqref{eqs:abc} back into the energy balance \eqref{eq:energy} we find
\begin{equation}\label{eq:final}
\delE(t+\step) - \delE(t)  \leq -\step(t) \big(1-\deg_+(t)\cdot\step(t)\big)\frac{1}{M}\sum_{i,j} \phi_{ij}(t)|\bv_j(t)-\bv_i(t)\big|^2m_im_j.
\end{equation}
Observe that we now pay attention to the time dependence on the right; in particular, the possibly variable time step, $\step=\step(t)$,  and the  time-dependent communication weights, $\phi_{ij}(t)=\phi(\bx_i(t),\bx_j(t))$.\newline
We let ${\mathbb A}(t)$ denote the $N\times N$ \emph{adjacency matrix} ${\mathbb A}(t)=\{\phi_{ij}(t)\}$  encoding the edges of communication at time $t$,  and define $\Delta_{{\mathbf m}}{\mathbb A}(t)$ is the \emph{weighted graph Laplacian} 
\begin{equation}\label{eq:weighted}
\left(\Delta_{{\mathbf m}}{\mathbb A}\right)_{\alpha\beta}=\left\{\begin{array}{ll}
-\ds \phi_{\alpha\beta}\sqrt{m_\alpha m_\beta}, & \alpha\neq \beta\\ \\
\ds \sum_{\gamma\neq \alpha}\phi_{\alpha\gamma}m_{\gamma} & \alpha=\beta.
\end{array}\right.
\end{equation}
The weighted graph Laplacian, weighted by the masses ${\mathbf m}=(m_1,\ldots, m_N)$, has real eigenvalues, $\lambda_1=0 \leq \lambda_2 \leq \ldots \lambda_N$. This generalizes the usual notion of graph Laplacian, e.g.,\cite{Mer94,Chu97}, corresponding to the case of uniform weight, $m_i={\mathcal O}(\nicefrac{1}{N})$. We now  summarize the computations above, quantifying the decay of energy fluctuations in terms of the spectral gap, $\lambda_2\big(\Delta_{\mathbf m}{\mathbb A}(t)\big)$.

\begin{theorem}[{\bf Decay of energy fluctuations}]\label{thm:main1} Consider the C-S dynamics \eqref{eq:CS} with time-steps small enough such that 
\begin{equation}\label{eq:CFL}
\step(t) \cdot \max_i\sum_j \phi_{ij}(t)m_j \leq \hf.
\end{equation}
Then the following bound of energy fluctuations holds
\begin{equation}\label{eq:result}
\delE(t_n) \leq exp\Big\{-\sum_{k=0}^{n-1} \lambda_2(t_k)\step(t_k)\Big\}\delE_0, \quad \lambda_2(t)=\lambda_2\big(\Delta_{\mathbf m}{\mathbb A}(t)\big), \ \  t_{k+1}=t_k+\step(t_k).
\end{equation}
\end{theorem}
\begin{proof}
We return to  the energy fluctuations bound \eqref{eq:final}. It remains to relate the enstrophy on the right of \eqref{eq:final} to the energy fluctuations on the left.
To this end, we use the following sharp lower bound on the enstrophy \cite[\S3]{HT21}, expressed in terms of its spectral gap $\lambda_2(t)=\lambda_2\big(\Delta_{{\mathbf m}}{\mathbb A}(t)\big)$,
\begin{equation}\label{eq:laplacian}
\sum_{i,j} \phi_{ij}(t)|\bv_j(t)-\bv_i(t)\big|^2m_im_j \geq \frac{\lambda_2(t)}{M} \sum_{i,j} |\bv_j(t)-\bv_i(t)\big|^2m_im_j, \quad \lambda_2= \lambda_2(\Delta_{{\mathbf m}}{\mathbb A}).
\end{equation}
Inserted into \eqref{eq:final}, the time-step restriction \eqref{eq:CFL} and \eqref{eq:laplacian} yield
\[
\begin{split}
\delE(t+\step) & \leq \delE(t) - \frac{\step(t)}{2}\lambda_2(t)\frac{1}{M^2}\sum_{i,j} |\bv_j(t)-\bv_i(t)\big|^2m_im_j \\
 & \ \ = \Big(1-\step(t) \lambda_2(t)\Big)\delE(t) \leq e^{-\step(t)\lambda_2(t)}\delE(t),
 \end{split}
\]
and \eqref{eq:result} follows.
\end{proof}

\noindent
\begin{remark}[{\bf Graph connectivity}]
The  weighted graph Laplacian $\Delta_{{\mathbf m}}{\mathbb A}$ is symmetrizable, with real eigenvalues $\lambda_1=0\leq \lambda_2 \leq \ldots \leq \lambda_N$. The weighted Poinacr\'{e} inequality \eqref{eq:laplacian} provides a sharp lower bound on the enstrophy in terms of the spectral gap 
$\lambda_2(\Delta_{{\mathbf m}}{\mathbb A})>0$, which reflects the connectivity of the weighted graph $(\Vgraph,\Egraph)$, where vertices of $\Vgraph$ tag the positions $\{\bx_i\}$ and the edges $\Egraph$ quantify the connections $\{\phi_{ij}\}$. The intricate aspect here is the interplay between the  graph which is time dependent, $(\Vgraph(t),\Egraph(t))$, hence its various properties are dictated  by the alignment dynamics on the graph, and at the same time, as we observe in theorem \ref{thm:main1}, the fluctuations of alignment dynamics are dictated by the connectivity of the underlying graph.\newline
The spectral gap,  $\lambda_2(\Delta_{\mathbf m}{\mathbb A})$,  generalizes the usual notions of graph connectivity in  terms of the \emph{Fiedler number} in case of uniform weights $m_i\equiv \nicefrac{1}{N}$. In particular, we point out that the weighted Poincar\'{e} bound \eqref{eq:laplacian} depends only on the total mass $M$ but otherwise is independent of the condition number, $\ds \frac{\max_i m_i}{\min_i m_i}$.
\end{remark}

Theorem \ref{thm:main1} describes the long time behavior of a fully-discrete C-S dynamics \eqref{eq:CS} under  a general setup based on symmetric communication kernel, $\phi_{ij}=\phi_{ji}$, which involves variable spatial weights $m_i$ and variable time stepping, $\step=\step(t_k)$ satisfying a CFL-like time-step restriction \eqref{eq:CFL}. In particular, letting $\max_k \step(t_k) \rightarrow 0$ we recover the semi-discrete CS model
\begin{equation}\label{eq:SDCS}
\left\{ \ \ \begin{split}
\frac{\textnormal{d}}{\textnormal{d}t}{\bx}_i(t)&=\bv_i(t)\\
\frac{\textnormal{d}}{\textnormal{d}t}{\bv}_i(t)&=\sum_{j\in {\mathcal N}_i} m_j(t)\phi_{ij}(t)(\bv_j(t)-\bv_i(t)). 
\end{split}\right.
\end{equation}
and theorem \ref{thm:main1} tells that
\begin{equation}\label{eq:SDresult}
\delE(t) \leq exp\Big\{-\int_0^t \lambda_2(s)\textnormal{d}s\Big\}\delE_0.
\end{equation}

\subsection{Fluctuations revisited--- $\ell^\infty$-diameter of fluctuations}
We measure the  fluctuations of velocities in terms of the  the $\ell^\infty$-diameter of the collection of velocities $\{\bv_i\}$
\[
\delV(t):=\max_{i,j}|\bv_i(t)-\bv_j(t)|.
\]
It will be convenient  to trace the scalar components which form this diameter. To this end we fix an arbitrary unit vector\footnote{The vector norm $|\cdot|$ is assumed to have it dual $|\bo|_*=\sup_{|\bv|=1}\langle \bv,\bo\rangle$.}, $|\bo|_*=1$. Since  
$|\bv|=\max_{|\bo|_*=1} \langle \bv, \bo\rangle$ then
\[
\delV(t):=\max_{|\bo|_*=1}\max_{p,q}\big(\v_p(t)-\v_q(t)\big), \qquad \v_p(t):=\langle \bv_p(t),\bo\rangle.
\]
We now trace the decay of these scalar components of velocity fluctuations, considering an arbitrary $(p,q)$ pair,  $v_p(t)-v_q(t)$, where as before we suppress the  dependence on time $t$ on the right, 
\[
\begin{split}
\v_p(t+\step)&-\v_q(t+\step) \\
 & =\v_p-\v_q + \step \sum_j m_j\phi_{pj}(\v_j-\v_p)-\step \sum_j m_j\phi_{qj}(\v_j-\v_q) \\
& = \big(1-\step \sum_j m_j\phi_{pj}\big)\v_p -\big(1-\step \sum_j m_j\phi_{qj}\big)\v_q
+ \step\sum_j m_j\phi_{pj}\v_j - \step \sum_j m_j\phi_{qj}\v_j\\
& = \big(1-\step \sum_j m_j\phi_{pj}\big)\v_p -\big(1-\step \sum_j m_j\phi_{qj}\big)\v_q \\
& \hspace*{2cm}+ \step\sum_j m_j(\phi_{pj}-c_j)\v_j - \step \sum_j m_j(\phi_{qj}-c_j)\v_j.
\end{split}
\]
In the last step we introduced arbitrary scalars $c_j$'s --- their contribution to the last two terms on the right cancel out. By the CFL condition \eqref{eq:CFL}, the first two parenthesis on the  right are positive. We now set
$c_j:=\min_{p,q}\{\phi_{pj},\phi_{qj}\}$ ---  with this choice the last two parenthesis on the right are also non-negative. Hence, if we let $\v_+$ and $\v_-$ denote the extreme values $ \v_+ :=\max_p \v_p$ and $\v_-:=\min_q \v_q$ we conclude
\[
\begin{split}
\v_p(t+\step)&-\v_q(t+\step) \\
 & \leq \big(1-\step \sum_j m_j\phi_{pj}\big)\v_+ -\big(1-\step \sum_j m_j\phi_{qj}\big)\v_- \\
 &   \hspace*{2cm} + \step\sum_j m_j(\phi_{pj}-c_j)\v_+ - \step \sum_j m_j(\phi_{qj}-c_j)\v_-\\
  & = \v_+ - \v_- -\step \sum_j m_jc_j(\v_+-\v_-) 
   =\Big(1-\step \sum_j m_jc_j\Big)\Big(\max_p \v_p-\min_q \v_q\Big) \\
  & =\big(1-\step \erg({\mathbb A})\big) \max_{p,q}\big(\v_p-\v_q), \qquad \erg({\mathbb A}):= \sum_j m_j\min_{p,q}\{\phi_{pj},\phi_{qj}\}.
  \end{split}
\]
Since $(p,q)$ is an arbitrary pair we conclude
\begin{equation}\label{eq:ergodicity}
\delV(t+\step)= \max_{|\bo|_*=1}\max_{p,q}\big(\v_p(t+\step)-\v_q(t+\step)\big)\leq \big(1-\step \erg({\mathbb A}(t) \big)\delV(t).
\end{equation}

\begin{theorem}[{\bf Decay of uniform fluctuations}]\label{thm:main2} Consider the C-S dynamics \eqref{eq:CS} with time-steps small enough such that \eqref{eq:CFL}  holds
\[
\step(t) \cdot \max_i\sum_j \phi_{ij}(t)m_j(t) \leq \hf.
\]
Then the following bound of the diameter of fluctuations holds
\begin{equation}\label{eq:kapparesult}
\delV(t_n) \leq exp\Big\{-\sum_{k=0}^{n-1} \erg\big({\mathbb A}(t_k)\big)\step(t_k)\Big\}\delV_{\!\!0}, \quad \erg\big({\mathbb A}(t)\big)=\sum_j m_j(t) \min_{p,q}\{\phi_{pj}(t),\phi_{qj}(t)\}.
\end{equation}
\end{theorem}

We emphasize that the bound \eqref{eq:kapparesult} applies to C-S dynamics with  general communication, $\{m_j(t)\phi_{ij}\}$, which need \emph{not} be symmetric, as it allows for time-dependent masses.
In particular, it applies to both the Cucker-Smale alignment model with symmetric interactions, \eqref{eq:equiCS}, $m_j=\nicefrac{1}{N}$ and the Motsch-Tadmor alignment model with non-symmetric, time-dependent interactions \eqref{eq:MT}, $m_i(t)=\frac{1}{L}\deg_i(t)$.\newline
In  case of a uniform-in-time lower bound $\erg\big({\mathbb A}(t_k)\big)\geq \eta>0$, e.g., see the particular case of all-to-all connectivity in \eqref{eq:refined}, we end up with the exponential decay
\begin{equation}\label{eq:etaresult}
\delV(t) \leq e^{-\eta t}\delV_{\!\!0}, \qquad \qquad \eta= \min_{t_k}\sum_j m_j(t_k) \min_{p,q}\{\phi_{pj}(t_k),\phi_{qj}(t_k)\}.
\end{equation}

\begin{remark}[{\bf Spectral gap vs. coefficient of ergodicity}]
The role of spectral gap in the present context of connectivity of graph goes back to Fiedler \cite{Fie73,Fie89}. In the case of equal weights, the so-called Fiedler number $\lambda_2(\Delta{\mathbb A})>0$ quantifies the algebraic connectivity of the graph $(\Vgraph,\Egraph)$ supported at vertices $\Vgraph=\{i : \, \bx_i\}$ with weighted edges $\Egraph=\{(i,j) : \, \phi_{ij}>0\}$.\newline
The inequality \eqref{eq:laplacian} is sharp in the sense that
\[
\frac{1}{M}\lambda_2(\Delta_{\mathbf m}{\mathbb A})
=\min \frac{\sum_{i,j} \phi_{ij}|\bv_i-\bv_j|^2 m_im_j}{\sum_{i,j} |\bv_i-\bv_j|^2m_im_j}.
\]
The obvious bound  that follows, 
\begin{equation}\label{eq:minphi}
\lambda_2(\Delta_{\mathbf m}{\mathbb A})
 \geq M\min_{i,j}\phi_{ij},
 \end{equation}
 shows that $\phi_{ij}(t)>0 \ \leadsto \ \lambda_2(\Delta_{\mathbf m}{\mathbb A})(t)>0$. This is the scenario  of a global, all-to-all connectivity between every pair of agents. The bound \eqref{eq:minphi} is not sharp: we may have certain edges vanish while still maintaining a connected graph, that is, the strict inequality $\lambda_2 > M\min_{ij}\phi_{ij}=0$ holds. A positive coefficient of ergodicity  allows more general scenarios, in which   pairs of agents,  positioned at say $\bx_p$ and $\bx_q$,  may lack direct communication, $\phi_{pq}=0$, but they still communicate through  an intermediate  agent positioned at $\bx_k$.  That is,  for each $(p,q)$ there exists (at least) one agent positioned at $\bx_k, \ k = k(p,q)$, which is the ‘go between’ agent so that\footnote{Of course the special case $k=p$ recovers the direct pairwise communication.} $\min\{\phi_{pk}, \phi_{qk}\} > 0$. This one-layer of communication is captured by the refined lower-bound
\begin{equation}\label{eq:refined}
\lambda_2(\Delta_{\mathbf m}{\mathbb A})
 \geq \erg({\mathbb A})=\sum_j m_j \min_{p,q}\{\phi_{pj},\phi_{qj}\} \geq M\min_{i,j}\phi_{ij}.
\end{equation}
\end{remark}

The estimate \eqref{eq:ergodicity} in its $\ell^1$-dual form for  goes back to Dobrushin \cite{Dob56}, quantifying the contractivity of  column-stochastic matrices in terms of  the so-called coefficient of ergodicity, denoted here $\erg({\mathbb A})$, \cite{IS11}. It was revisited in many follow-up works, e.g.,  its used to quantify
the relative entropy in discrete Markov processes \cite{CDZ93,CIR93} scrambling in models of opinion dynamics \cite{Kra00,Kra15} and flocking dynamics \cite[\S2.1]{MT14a}.

\subsection{Energy fluctuations revisited --- time dependent masses}\label{sec:MT}
The study of long time behavior based on $\ell^\infty$-diameter of velocity fluctuations enjoyed the advantage of addressing time-dependent masses. In contrast, our study of energy fluctuations in section \ref{sec:energy} was restricted to constant masses. Here we observe that the proof of theorem \ref{thm:main1} can be adapted to include the case of time-dependent masses, $m_i=m_i(t)$. Indeed, the time variability of the masses enters at precisely in two places:  the time invariant  total momentum in \eqref{eq:momentum}
\begin{subequations}\label{eqs:equalities}
\begin{equation}\label{eq:equa}
{\mathscr M}(t+\step)={\mathscr M}(t), \qquad {\mathscr M}(t):=\sum_i m_i(t)\bv_i(t),
\end{equation}
and the evaluation of the incremental energy fluctuations \eqref{eq:abc}
\begin{equation}\label{eq:equb}
\delE(t+\step) - \delE(t) = \sum_i m_i(t+\step) |\bv_i(t+\step)|^2 -\sum_i m_i(t)|\bv_i(t)|^2.
\end{equation}
\end{subequations}
To pursue our line of proof when $m_i=m_i(t)$, the momentum ${\mathscr M}(t+\step)$ and energy fluctuations $\delE(t+\step)$ need to be weighted by $m_i(t+\step)$ rather than $m_i(t)$. Thus, the two qualities above admit the additional terms 
\[
\sum_i |m_i(t+\step)-m_i(t)|\cdot |\bv_i(t+\step)|,
\]
 and, respectively, 
 \[
 \sum_i |m_i(t+\step)-m_i(t)|\cdot |\bv_i(t+\step)|^2,
 \]
so one needs to control the  incremental  changes  $\sum_i |m_i(t+\step)-m_i(t)|$.
Consider the example of  the M-T model \eqref{eq:MT} with metric kernel
$\phi(\bx,\bx')=\varphi(|\bx-\bx'|)$ where the time-dependent masses are then given by the degrees
\[
m_i(t)=\frac{1}{L}\deg_i(t)=\sum_j \phi_{ij}(t), \qquad \phi_{ij}(t)= \varphi(|\bx_i(t)-\bx_j(t)|).
\]
Assume that $\varphi$ is a \emph{smooth} metric communication kernel satisfying a localized Lip bound in the sense that
\[
|\varphi(r)-\varphi(s)| \leq C_1\max\{|\varphi(r)|,|\varphi(s)|\}|r-s|.
\]
 For this large class of  localized Lip bounded $\varphi$'s (which includes for example, $\varphi(r)=(1+r)^{-\beta}$ with $C_1=\beta$), we have
\[
\begin{split}
|m_i(t+\step)-m_i(t)| & \leq \sum_j \big(\varphi(|\bx_i(t+\step)-\bx_j(t+\step)|)-
\varphi(|\bx_i(t)-\bx_j(t)|)\big) \\
 & \leq C_1\sum_j\max\{\phi_{ij}(t+\step),\phi_{ij}(t)\}
\cdot \big|\big(\bx_i(t+\step)-\bx_j(t+\step)\big)-\big(\bx_i(t)-\bx_j(t)\big)\big|,
\end{split}
\]
and hence
\[
\sum_i |m_i(t+\step)-m_i(t)|\cdot|\bv_i(t+\step)| \\
  \leq 
C_1 \sum_{i,j} \max\{\phi_{ij}(t+\step),\phi_{ij}(t)\} \cdot \tau |\bv_i(t)-\bv_j(t)|   \cdot |\bv_i(t+\step)|.
\]
Now, using a uniform bound on the velocities,  $\max_i |\bv_i(t+\step)|\leq C_2$,  the exponential decay of $\delV(t)$,  \eqref{eq:etaresult}, and recalling $\phi_{ij}= L\widetilde{\phi}_{i,j}m_im_j$ in \eqref{eq:MTS}, we find
\[
\begin{split}
\sum_i |m_i(t+\step)&-m_i(t)|\cdot |\bv_i(t+\step)|  \\
 & \leq C_1C_2 \tau \cdot \delV(t) \sum_{i,j} \max\big\{\widetilde{\phi}_{ij}(t) m_i(t) m_j(t),  \widetilde{\phi}_{ij}(t+\step) m_i(t+\step) m_j(t+\step)\big\}
 \\
 & \leq C'C_2 \tau e^{-\eta t}, \qquad \qquad C':=C_1M^2\cdot\max |\phi|\cdot\delV_{\!\!0}.
\end{split}
\]
Hence, the  equalities \eqref{eqs:equalities} in the case of constant masses are now replaced by the corresponding
\begin{subequations}
\begin{equation}
\left|{\mathscr M}(t+\step)-{\mathscr M}(t)\right|\leq C'\tau e^{-\eta t},
\end{equation}
 and, respectively,
 \begin{equation}
 \Big|\Big(\delE(t+\step) -\sum_i m_i(t+\step)|\bv_i(t+\step)|^2 \Big)- \Big( \delE(t) - \sum_i m_i(t)|\bv_i(t)|^2\Big) \Big| \leq C'C_2^2\tau e^{-\eta t}.
 \end{equation}
 \end{subequations}
 Thus, presence of smoothly varying  time-dependent masses, accounts for additional terms which have a bounded accumulated effect.
 One can then study the long time behavior based on energy fluctuations in the presence of time-dependent masses, similar to our  discussion in the next section, of flocking/swarming phenomena  with constant masses.
 
\section{Flocking and Swarming}
The phenomena of \emph{flocking} or \emph{swarming} require  the emergence of coordinated long time behavior of velocities, while the crowd of agents remains contained within finite diameter
\begin{equation}\label{eq:diameter}
D(t):=\max_{i,j}|\bx_i(t)-\bx_j(t)| \leq D_+ <\infty.
\end{equation}
The emerging behavior of velocities in intimately linked to the decay bounds of energy fluctuations.
Indeed,  \eqref{eq:result} and its corresponding semi-discrete \eqref{eq:SDresult} imply that \emph{if} the weighted graph of communication remains sufficiently strongly connected in the sense that $\lambda_2(t)$ has diverging tail, then by \eqref{eq:mean}
\begin{equation}\label{eq:heavy}
\begin{split}
\int^\infty \lambda_2(s)\textnormal{d}s=\infty \  
  \ \leadsto \  \ \sum_i |\bv_i(t)-\overline{\bv}(t)|^2m_i \leq exp\Big\{{\small -\int_0^t \lambda_2(s)\textnormal{d}s}\Big\}\delE_0\stackrel{t\rightarrow \infty}{\longrightarrow}0.
 \end{split}
\end{equation}
In particular, since the mean velocity is an invariant of the flow, 
\[
\overline{\bv}(t):=\frac{1}{M}\sum_j m_j\bv_j(t)=\overline{\bv}_0,
\]
\eqref{eq:heavy} tells us that a heavy-tailed $\lambda_2(t)$ implies  the long time behavior  of  the velocities  that align along the initial  mean, $\bv_i(t) \stackrel{t\rightarrow \infty}{\longrightarrow} \overline{\bv}_0$

 \begin{remark}[{\bf Emerging velocity in presence of time-dependent masses}] In case of constant masses, the mean velocity $\overline{\bv}_(t)$ remains invariant in times, and the decay of velocity fluctuations implies the emergence of  $\overline{\bv}_0$ as the limiting velocity.  The presence of time-dependent masses, however, leaves open the question of what is the emerging velocity. Thus, for example, in case of the M-T dynamics \eqref{eq:MT}, we expect  that velocities will align along the corresponding mean $\overline{\bv}$
 \[
 |\bv_i(t) -\overline{\bv}(t)| \stackrel{t\rightarrow \infty}{\longrightarrow}0,
 \qquad \overline{\bv}(t):= \frac{1}{\sum_j \deg_j(t)}\sum_j \deg_j(t)\bv_j(t). 
 \]
 The question is if and when the emerging  \emph{limiting} velocity,
 $\ds \lim \limits_{t\rightarrow \infty} \frac{1}{\sum_j \deg_j(t)}\sum_j \deg_j(t)\bv_j(t)$, exists.
  \end{remark}

\subsection{Long-range interactions} 
But when does $\lambda_2(t)$ satisfy the `heavy-tail' condition sought in  \eqref{eq:heavy}? this  is a central question for tracing the phenomenon of flocking. It was addressed in many references, starting with the original \cite{CS07a,CS07b} followed by \cite{HT08,HL09b}; see \cite{CFTV10,Shv21} and the references therein. A definitive answer is provided in case of \emph{long-range kernels}, 
\begin{equation}\label{eq:long}
\phi(\bx,\bx')\gtrsim \frac{1}{(1+|\bx-\bx'|)^\beta}, \quad \beta>0.
\end{equation}
In this case we bound the tail of the spectral gap,
$\lambda_2(t)=\lambda_2\big(\Delta_{\mathbf m}{\mathbb A}(t)\big)$,
\begin{equation}\label{eq:connect}
\lambda_2\big(\Delta_{\mathbf m}{\mathbb A}(t)\big) 
  \geq M\min \phi_{ij}(t) \gtrsim \frac{M}{\big(1+D(t)\big)^\beta}\geq
\frac{M}{(1+D_0 + \delV_{\!\!0}\cdot t)^\beta}.
\end{equation}
The first inequality on the right  follows from \eqref{eq:minphi}, the second follows from \eqref{eq:long} and the third follows from a uniform bound on the diameter of velocities\footnote{\label{foot:max}The result follows without appealing to the bound on diameter of velocities \eqref{eq:ergodicity}. Instead, a simpler  maximum principle argument follows from the CFL condition \eqref{eq:CFL}, 
\[
|\bv_i(t+\step)|\leq \Big(1-\step\sum_j\phi_{ij}m_j\Big)|\bv_i(t)|+\step\sum_j \phi_{ij}m_j|\bv_j(t)| \leq  \max_j|\bv_j(t)| \leq \ldots \leq v_+(0), \quad v_+(0)=\max_i |\bv_i(0)|,
\]
and integration of  \eqref{eq:CS}${}_1$, and likewise, \eqref{eq:SDCS}${}_1$ in the semi-discrete case, imply $D(t) \leq D_0+2v_+(0)\cdot t$.},
\[
D(t)\leq D_0 + \int^t \delV(s)\textnormal{d}s \leq D_0+\delV_{\!\!0}\cdot t, 
\]
  and hence the heavy-tailed bound, $\lambda_2(t)\gtrsim (1+t)^{-\beta}$ for $\beta\leq 1$. We conclude that  the C-S dynamics \eqref{eq:CS} with long-range communication \eqref{eq:long},  $\beta\leq 1$, admits \emph{unconditional flocking} 
\[
\sum_i |\bv_i(t)-\overline{\bv}_0|^2m_i \lesssim
\left\{\begin{array}{ll}
exp\big\{-\frac{M}{(1-\beta)\delV_{\!\!0}}\big(1+D_0+\delV_{\!\!0}\cdot t\big)^{1-\beta}\big\} & \beta<1\\ \\
\big(1+D_0+\delV_{\!\!0}\cdot t\big)^{\ds -\nicefrac{M}{\delV_{\!\!0}}} & \beta=1
\end{array}\right\} \stackrel{t\rightarrow \infty}{\longrightarrow}0.
\]
We can now use a bootstrap argument --- the fractional exponential decay of the fluctuations of order $1-\beta>0$ implies that the diameter remains uniformly bounded and hence  uniform bounded connectivity
\[
D(t)\leq D_0 + \int^t \delV(s)\textnormal{d}s \leq D_+ \ \ \leadsto \ \ \lambda_2(t) \geq \eta:=\frac{1}{(1+D_+)^\beta}.
\]
 Revising \eqref{eq:connect} with a finite diameter $\leq D_+$, yields the improved exponential bound
 $\lesssim e^{-\eta t}$. 
 A similar argument applies in the borderline case of $\beta=1$:  clearly, if $\delV_{\!\!0}<M$ then the finite tail of $\lesssim (1+t)^{-\nicefrac{M}{\delV_{\!\!0}}}$ will lead to a finite diameter; and indeed, since $\delV(t)$ is decaying, we will eventually reach the threshold $\delV(t_c)<M$ and exponential decay follows thereafter. We summarize.
 \begin{theorem}[{\bf Flocking/swarming with long range kernels}]\label{thm:flocking1}
 Consider the C-S dynamics \eqref{eq:CS} driven by long-range kernel \eqref{eq:long}, $\beta\leq1$. Then the crowd of agents has finite support $D_+$ and  there is exponential decay
 of fluctuations around the mean velocity,
\begin{equation}\label{eq:exp}
\sum_i |\bv_i(t)-\overline{\bv}_0|^2m_i \lesssim
e^{-\eta t} \delE_0, \qquad \eta=\frac{1}{(1+D_+)^\beta}.
\end{equation}
\end{theorem}
The precise exponential  bound, $\eta=\eta(\beta,D_0,v_+)$, was captured in \cite{HL09b} using an elegant argument based on a proper Liapunov functional 
for C-S with metric kernel and uniform masses.\newline

\begin{remark}[{\bf No uniform bound}] We distinguish between two types of bounds on  velocity fluctuations ---  the $\ell^2$-based energy fluctuations, theorem \ref{thm:main1}, and the $\ell^\infty$ bounds, theorem \ref{thm:main2}  or at least the uniform bound on velocities, see footnote \ref{foot:max}. Suppose we try to pursue  a purely $\ell^2$-based argument 
  for flocking behavior. The energy bound \eqref{eq:result} implies
  the uniform-in-time bound
  $\max_i|\bv_i(t)-\overline{\bv}_0| \leq C\sqrt{N}$ which in turn yields a bound on the diameter $D(t) \leq D_0+2C\sqrt{N}t$. We now use the same bootstrap argument as before to find a uniform-in-time bound on the diameter $D(t) \lesssim D_+(N):=N^{\frac{\beta}{2(1-\beta)}}$. We conclude an exponential flocking of rate
  \[
  \sum_i |\bv_i(t)-\overline{\bv}_0|^2m_i \lesssim D_+(N)
  e^{-D_+(N)t}, \qquad D_+(N)=N^{\frac{\beta}{2(1-\beta)}}.
  \]
  As expected, the fluctuations bound grows with $N$. However, the point to note here is that the exponential decay in time enforces exponential alignment  bound, uniform in $t$ \emph{and} $N$  when $t\gg  N^{\frac{\beta}{2(1-\beta)}}$. For example,  $\beta=\nicefrac{1}{4}$ requires a moderate time of  $t\gg N^{\nicefrac{1}{6}}$ before  exponential decay  takes place.
  \end{remark} 

Theorem \ref{thm:flocking1} was derived based on considerations of energy fluctuations. Similarly,  we  can proceed using the $\ell^\infty$-diameter fluctuations of theorem \ref{thm:main2}. Its semi-discrete limit $\max_k\tau(t_k)\rightarrow 0$ reads
\[
\max_{i,j}|\bv_i(t)-\bv_j(t)|\leq exp\Big\{-\int_0^t \sum_j m_j(s)\min_{p,q}\{\phi_{pj}(s),\phi_{qj}(s)\}\textnormal{d}s\Big\}\delV_{\!\!0}. 
\]
Here, we generalize theorem \ref{thm:flocking1} to the case of time-dependent masses. Using a bootstrap argument as before we end up with
 \begin{theorem}[{\bf Flocking/swarming with long range kernels --- time-dependent masses}]\label{thm:flocking2}
 Consider the C-S dynamics \eqref{eq:CS} with possibly time-dependent masses, $m_i=m_i(t)$, driven by long-range kernel \eqref{eq:long}, $\beta\leq1$. Then the crowd of agents has finite support $D_+$ and  there is exponential decay
 of fluctuations of velocities,
\begin{equation}\label{eq:exp2}
\max_{i,j} |\bv_i(t)-\bv_j(t)| \lesssim
exp\Big\{-\ds \eta \int_0^t \sum_j m_j(s) \textnormal{d}s\Big\} \delV_{\!\!0}, \qquad \eta=\frac{1}{(1+D_+)^\beta}.
\end{equation}
\end{theorem}
Observe that in the example of M-T model \eqref{eq:MT} the scaling alluded in footnote \ref{foot:MT}, $M={\mathcal O}(1)$, implies $\ds \int_0^t \sum_j m_j(s) \textnormal{d}s\geq Ct$ and hence we end up with an exponential decay 
$e^{-\eta C t}$.

\smallskip
The arguments that led to theorem \ref{thm:flocking1} and the new theorem \ref{thm:flocking2} demonstrate a rather general methodology for studying flocking, swarming and more general emerging phenomena in alignment based dynamics. It consists of two main ingredients: 

\begin{itemize}
 \item  Decay of energy fluctuations. This is tied to spectral analysis of the  dynamic graph $\big(\Vgraph(t),\Egraph(t)\big)$. In typical cases, the dynamics is equipped with an intrinsic `energy' and energy fluctuations. 

\item Bound on the velocities --- either a uniform bound on velocities or on $\ell^\infty$-diameter of velocities fluctuations. In either case, the purpose is to trace the size of the spatial diameter and show that the crowd does not disperse,
 $D(t) \leq D_+$. In general, this is the more intricate bound to prove.
  \end{itemize}
 
 \noindent
 As an example we mention alignment dynamics with external forcing \cite{ST20a}. Other examples include C-S dynamics with \emph{matrix} communication kernels, and C-S dynamics in which both, alignment and attraction, take place. We continue with this example in the context of \emph{anticipation dynamics}.

\subsection{From anticipation to Cucker-Smale dynamics}
Particles are driven by the external forces induced by the environment and/or  by other particles. The dynamics of social particles, on the other hand,   is driven by  \emph{probing} the environment  --- living organisms, human interactions and sensor-based agents have \emph{senses and sensors}, with which they actively probe the environment (and hence they are commonly viewed as `active particles' \cite{BDT17-19}). A distinctive feature of active particles in probing the environment is \emph{anticipation} --- the dynamics  is not driven instantaneously, but reacts to positions $\bx^\step(t):=\bx(t)+\step \bv(t)$, \emph{anticipated} at $t+\step$, where $\step>0$ is an anticipation time increment.  A general framework for anticipation dynamics, driven by pairwise interactions induced by  radial  potential $U=U(r)$, reads
\vspace*{-0.2cm}
\[
\left\{\quad
\begin{split}
\frac{\textnormal{d}}{\textnormal{d}t}{\bx}_i(t)&=\bv_i(t)\\
\frac{\textnormal{d}}{\textnormal{d}t}{\bv}_i(t)&= -\frac{1}{N}\sum_{j=1}^N\nabla U(|\bx^\step_i(t)-\bx^\step_j(t)|), \qquad \bx^\tau_k(t):=\bx_k(t)+\tau \bv_k(t).
\end{split}
\right.
\]
The alignment is encoded here in the anticipated time --- indeed, expanding 
the RHS in (the assumed small) $\step$ we obtain, \cite{ST21a},
\begin{equation}\label{eq:system-anticipation}
\frac{\textnormal{d}}{\textnormal{d}t} \bv_i(t)=\overbrace{-\frac{1}{N}\sum_{j} \nabla U(|\bx_j-\bx_i|)}^{\text{repulsion+attraction}} + \overbrace{\frac{\tau}{N}\sum_{j\in {\mathcal N}_i} \bbphi_{ij}(\bv_j-\bv_i)}^{\text{alignment}}, \qquad \Phi_{ij}=D^2U(|\bx_i-\bx_j|).
\end{equation}

\vspace*{0.5cm}\noindent
Thus, we \emph{derive} a general class of 3Zone models  \eqref{eq:system-anticipation}, where the first terms on the right account for repulsion/attraction, depending whether their scalar amplitudes $U'_{ij}:=U'(|\bx_i-\bx_j|)<0$ or, respectively,  $U'_{ij}>0$,  while the second term on the right accounts for  an alignment with \emph{matrix} coefficients, ${\bbphi}_{ij}= D^2U(|\bx_i-\bx_j|)$, see figure \ref{fig:Lennard-Jones}. 

\vspace*{0.5cm}
\begin{figure}[h]
    \begin{minipage}{120mm}
  \begin{center}
    \includegraphics[width=0.29\textwidth]{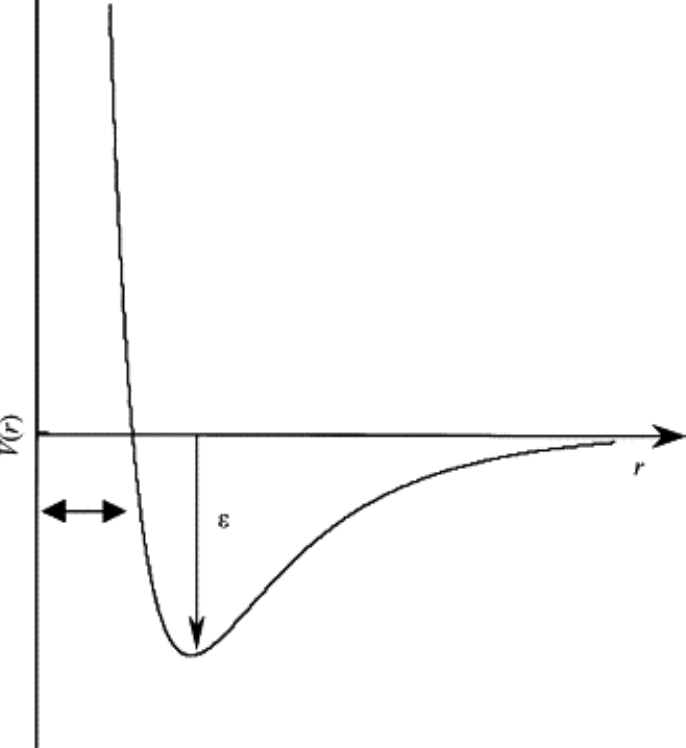}
   \caption{Potential $U(r)$}\label{fig:Lennard-Jones}  
      \end{center}
      \end{minipage}
            \end{figure}

\noindent
This leads us to  consider an even larger class of 3Zone models with repulsion/attraction induced by potential $U$ and  alignment term induced by a separate scalar symmetric kernel, $\phi_{ij}=\phi(\bx_i,\bx_j)$ (independent of $U$),

\begin{equation}\label{eq:anticipation}
\frac{\textnormal{d}}{\textnormal{d}t} \bv_i(t)=-\frac{1}{N}\sum_{j} \nabla U(|\bx_j-\bx_i|) + \frac{\tau}{N}\sum_{j: \phi_{ij}>0} \phi(\bx_i,\bx_j)(\bv_j-\bv_i).
\end{equation}
The special case of metric-based kernel $\phi_{ij}=\varphi(|\bx_i-\bx_j|)$ recovers the C-S dynamics \eqref{eq:equiCS}, $m_i\equiv\nicefrac{1}{N}$. For the special case  of anticipation \eqref{eq:system-anticipation} we have $\bbphi_{ij}\geq U''(|\bx_i-\bx_j|){\mathbb I}$. 

\medskip
The  energy fluctuations associated with \eqref{eq:anticipation} 
\[
{\delE}(t):=\frac{1}{2N}\sum_i |\bv_i(t)-\overline{\bv}|^2 + \frac{1}{2N^2}\sum_{i,j}U(|\bx_i(t)-\bx_j(t)|),
\]
 are dissipated due to alignment at a precise rate dictated by local velocity fluctuations,
\[
\frac{\textnormal{d}}{\textnormal{d}t} {\delE}(t) =-\frac{\step}{2N^2}\sum_{i,j}\phi_{ij}(t) |\bv_{i}-\bv_j|^2.
\]
We assume a smooth radial potential so that $U'(0)=U(0)=0$. It follows that the  class of  convex potentials and `fat-tailed'  kernels such that
\begin{equation}\label{eq:cond}
U''(|\bx_i-\bx_j|)+\phi_{ij}\gtrsim \langle |\bx_i-\bx_j|\rangle^{-\gamma}, \quad \gamma<\nicefrac{4}{5},
\end{equation}
  guarantee decay of energy fluctuations,  $\ds {\delE}(t) \leq C_0 exp\{-t^{\frac{4-5\gamma}{4-3\gamma}}\}$, which in turn implies asymptotic flocking  towards the average velocity, $\overline{\bv}=\frac{1}{N}\sum_j \bv_j$, \cite{ST21a}. Moreover, agents asymptotically congregate in space, forming a  traveling wave  dictated by the presence of an attractive  potential $U$, e.g., a quadratic $U$ leads to a limiting harmonic oscillator.
\cite{ST20a}. 

\smallskip
\paragraph{{\bf Open questions}} The arguments  above exclude two important features in collective dynamics: since \eqref{eq:cond} implies $U$ is increasing,  it does not address  the role of \emph{repulsion} in shaping the emergent behavior. The large time behavior of  2Zone repulsion-attraction models were discussed in, e.g., 
\cite{CDMBC07,FHK11,FH13,DDMW15,CCP17,CFP17}.
 The  corresponding  question for the full 3Zone  model, in which attraction, alignment and repulsion co-exist,  is mostly open. 
 
 Another key aspect is the long-range alignment sought by the 'heavy-tailed' kernels in \eqref{eq:cond}  which does not address the local character of self-organized dynamics. The long time collective behavior based on \emph{short-range} protocols hinges on the \emph{graph connectivity} of the crowd, realized by the \emph{adjacency matrix} ${\mathbb A}(t):=\{\phi_{ij}(t)\}$. Short-range interactions may lead to instability. This can be traced by the \emph{graph Laplacian} $\Delta {\mathbb A}(t)$:  while the initial configuration of the crowd is assumed to form one connected cluster expressed by the positivity of its spectral gap  $\lambda_2(\Delta {\mathbb A}(0)>0$, it may break down  into two or more disconnected clusters at a finite time when $\lambda_2(\Delta{\mathbb A}(t_c)=0$. Flocking analysis with short-range kerenls can be found in \cite{JE07,GPY16,Car17,MPT19,ST20b,DR21,Tad21}.
 
 \section{Large crowd dynamics}
The question of instability for a fixed number of $N$ agents governed by short-range alignment  is better addressed in the context of \emph{large crowd} dynamics of $N\gg 1$ agents. The latter  is realized by the empirical distribution
\[
f_N(t,\bx,\bv):=\frac{1}{M}\sum_i m_i\delta (\bx-\bx_i(t))\otimes \delta(\bv-\bv_i(t)).
\]
The large crowd dynamics is captured by its first two $\bv$ moments which are assumed to exist, \cite{Shv21}: 
\[
\rho(t,\bx)=\lim_{N\rightarrow \infty} \int f_N(t,\bx,\bv)\dv, \quad \rho\bu(t,\bx)=\lim_{N\rightarrow \infty} \int \bv f_N(t,\bx,\bv)\dv
\]
These are the  density and momentum which encode the macroscopic description of the agents based \eqref{eq:anticipation} (we abbreviate $\square=\square(t,\bx), \square'=\square(t,\by)$)

\begin{equation}\label{eq:hydro}
\left\{\begin{split}
\rho_t +\nabla_\bx\cdot (\rho \bu)&=0\\
(\rho\bu)_t + \nabla_\bx\cdot (\rho\bu\otimes \bu +{\mathbb P})&= 
\tau \int_{\by\in \Omega} \phi(\bx,\by)(\bup-\bu)\rho\rho'\dy
-\rho\nabla U * \rho(t,\bx)
\end{split}\right.
\end{equation}
There are several ingredients in the macroscopic description: the pressure (Reynolds stress) tensor, $\displaystyle {\mathbb P}(t,\bx):=\lim_{N\rightarrow \infty} \int (\bv-\bu)(\bv-\bu)^\top f_N(t,\bx,\bv)\dv$, encodes the second-order $\bv$ moments of $f_N$. The closure of \eqref{eq:hydro} is imposed by assuming a limiting distribution
at thermal equilibrium -- a Maxwellian. But there is no generic closure  in the present context of collective dynamics, since agents maintain their own  detailed energy balance  which is beyond the realm of collective motion. The two terms on the right capture scalar alignment and respectively attraction/repulsion induced by the potential $U$.  

\noindent
\subsection{Short-range interactions} 
For simplicity, we  ignore the role of attraction/repulsion and conclude with three examples which trace the 
 flocking behavior of the purely alignment hydrodynamics \eqref{eq:hydro} with $U\equiv 0$.
 
\smallskip\noindent
{\bf Non-vacuous dynamics}.
In the first example, we consider  the dynamics in the $2\pi$-torus driven by bounded short-range kernels, $\phi(\bx,\by)$, localized along the diagonal 
\begin{equation}\label{eq:bdd}
\frac{1}{\Lambda}\mathds{1}_{R_0}(|\bx-\by|) \leq \phi(\bx,\by) \leq {\Lambda}\mathds{1}_{2R_0}(|\bx-\by|), \quad R_0 \ll \pi.
\end{equation}
It follows that strong solutions with non-vacuous  density 
$\rho(t,\cdot)\gtrsim (1+t)^{-\nicefrac{1}{2}}$ flock around the limiting velocity $\overline{\bv}$ due to the  decay of energy fluctuations $\delE(t)\rightarrow 0$, \cite[Theorem 1.1]{ST20b}. As we noted in \cite[theorem 3.3]{Tad21}, the decay of energy fluctuations is independent of the specific closure of the pressure --- what really matters is the  non-vanishing density, the connectivity of the $\textnormal{supp}\rho(t,\cdot)$  which enables to propagate information of alignment.

\smallskip\noindent
{\bf Topological interactions}.
The non-vacuous lower bound $(1+t)^{-\nicefrac{1}{2}}$ is not sharp. As a second example we mention a topologically-based \emph{singular} communication  kernel, 
corresponding to \eqref{eq:whatismu}
\begin{equation}\label{eq:topo}
 \phi(\bx,\by) \sim  \mathds{1}_{R_0}(|\bx-\by|)\times \frac{1}{\textnormal{dist}^d_\rho(\bx,\by)},  
\end{equation}
which involves the density weighted distance $\displaystyle \textnormal{dist}_\rho(\bx,\by)=\Big(\int _{{\mathcal C}(\bx,\by)}\hspace*{-0.3cm}\d\rho(t,\bz)\Big)^{1/d}$.
it follows  that  smooth  solutions satisfying the relaxed  non-vacuous condition, $\rho(t,\cdot)\gtrsim (1+t)^{-1}$, must flock, \cite{ST20b}. Again, no vacuum is a key aspect  which  
enables the propagation of information: as long as no vacuous  islands are formed, alignment dictates flocking behavior.

\smallskip\noindent
{\bf Multi-species}. Our third example involves multi-species dynamics
\[
\left\{\begin{split}
(\rho_\alpha)_t +\nabla_\bx\cdot (\rho_\alpha \bu_\alpha)&=0\\
(\rho_\alpha\bu_\alpha)_t + \nabla_\bx\cdot (\rho_\alpha\bu_\alpha\otimes \bu_\alpha +{\mathbb P}_\alpha)&= 
\tau \int_{\by\in \Omega} \varphi_{\alpha\beta}(|\bx-\by|)(\bup_\beta-\bu_\alpha)\rho_\alpha\rho'_\beta\dy.
\end{split}\right.
\]
In this case, different species tagged by the identifiers $\alpha,\beta\in{\mathcal I}$, are distinguished by  their different protocol of communication with the environment of other species, $\phi_{\alpha\beta}$.
 In \cite{HT21} it was shown that if the different species maintain  non-vacuous densities $\rho_\alpha(t,\cdot)\gtrsim (1+t)^{-1}$ and if the communication array ${\mathbb A}(r):=\{\phi_{\alpha\beta}(r)\}$ forms a \emph{connected graph}, $\lambda_2 \big(\Delta{\mathbb A}(r)\big)\gtrsim (1+r)^{-\beta}$ with heavy-tail, $\beta<1$, then flocking follows, $\displaystyle \bu_\alpha \stackrel{t\rightarrow \infty}{\longrightarrow} \overline{\bu}:=\frac{1}{|{\mathcal I}|}\sum_{\alpha\in {\mathcal I}} \bu_\alpha$.

\end{document}